\documentclass{article}

\usepackage[english]{babel}
\usepackage[utf8x]{inputenc}
\usepackage[margin=2.5cm,a4paper]{geometry}
\usepackage{amsmath,tikz}
\usepackage{etoolbox}
\usetikzlibrary{matrix}
\usetikzlibrary{positioning}
\usepackage{amssymb}
\usepackage{amsthm}
\usepackage{moresize}
\usepackage{graphicx}
\usepackage[colorinlistoftodos]{todonotes}
\usepackage{mathtools}
\usepackage{hyperref}
\usepackage{cite}
\usepackage{fancyhdr}
\usepackage{float}
\usepackage{lscape}
\usepackage{enumerate}
\usepackage{listing}
\usepackage{graphics}
\usepackage{bbm}
\usepackage{wrapfig}
\usepackage{framed}
\usepackage{array}
\usepackage{listings}
\usepackage{hhline}
\usepackage{algorithm}
\usepackage{algpseudocode}
\usepackage{algorithmicx}
\usepackage{breqn}
\usepackage{verbatim}
\usepackage[toc,page]{appendix}
\usepackage{caption}
\usepackage{color}
\usepackage[normalem]{ulem}
\usepackage{subcaption}
\usepackage{nomencl}
\usepackage{multirow}
\usepackage{ wasysym } %female
\usepackage{leftidx}
\usepackage{times} %make font NTR
\usepackage{hyphenat}
\usepackage{ragged2e}
\usepackage{blindtext}
\usepackage{longtable}
\usepackage{setspace}
\usepackage[autostyle]{csquotes} 
\usepackage{listings}
\usepackage{longtable}
\usepackage{authblk}

\numberwithin{equation}{section}
\addto{\captionsenglish}{}

\DeclareMathOperator{\prob}{{\mathbb{P}}}

\DeclareMathOperator{\expec}{{\mathbb{E}}}

\DeclareMathOperator{\e}{{e}}
\newcommand{\dif}{\mathrm{d}}
\newcommand{\vep}{\varepsilon}

\newcommand{\B}{C}
\newcommand{\W}{\texttt{W}}

\newcommand{\T}{\mathcal{T}}
\newcommand{\M}{\mathcal{M}}
\renewcommand{\H}{\mathcal{H}}
\newcommand{\D}{\mathcal{D}}

\newcommand{\groot}{\varnothing}

\DeclareMathOperator{\Unif}{\texttt{Unif}}

\DeclareMathOperator{\Binom}{\texttt{Binom}}

\newcommand{\sss}{\scriptscriptstyle}

\newtheorem{theorem}{Theorem}[section]

\newtheorem{lemma}[theorem]{Lemma}

\newtheorem{example}[theorem]{Example}
\newtheorem{remark}[theorem]{Remark}

\newtheorem{corollary}[theorem]{Corollary}

\definecolor{halfgray}{gray}{0.55} % chapter numbers will be semi transparent .5 .55 .6 .0
\definecolor{purplish}{rgb}{0.41, 0.41, 0.64}
\definecolor{navy}{rgb}{0,0,0.52}
\definecolor{webbrown}{rgb}{.6,0,0}
\definecolor{Maroon}{cmyk}{0, 0.87, 0.68, 0.32}
\definecolor{NiceRed}{rgb}{0.41,0, 0}
\definecolor{Black}{cmyk}{0, 0, 0, 0}
\definecolor{Peacefulbackground}{rgb}{0.933333,0.9098,0.9098}
\definecolor{Poop}{RGB}{165,82,35}
\definecolor{darkgreen}{RGB}{50,205,50}

\hypersetup{
	colorlinks=true, 
	urlcolor=webbrown, linkcolor=NiceRed, citecolor=purplish, 
}

\begin{document}

\title{Scaling of the Cumulative Weights of the Invasion Percolation Cluster on a Branching Process Tree}
\author{Rowel G\"undlach}
\date{\today}
\affil{Eindhoven University of Technology}
\maketitle

\begin{abstract}
    We analyse the scaling of the weights added by  \textit{invasion percolation} on a branching process tree. This process is a paradigm model of {\em self-organised criticality}, where criticality is approach without a prespecified  parameter. In this paper, we are interested in the \textit{invasion percolation cluster}, obtained by performing invasion percolation for $n$ steps and letting $n\to\infty$. The volume scaling of the IPC was discussed in detail in \cite{Gundlach2022} and in this work, we extend this analysis to the scaling of the cumulative weights of the IPC.

    We assume a power-law offspring distribution on the branching process tree with exponent $\alpha$. In the regimes $\alpha>2$ and $\alpha\in(1,2)$, we observe a natural law-of-large-numbers result, where the cumulative weights have the same scaling as the volume, but converge to a different limit. In the case $\alpha<1$, where the weights added by invasion percolation vanish, the scaling regimes change significantly. For $\alpha\in(1/2,1)$, the weights scale exponentially but with a different parameter than the volume scaling, while for $\alpha\in(0,1/2)$ it turns out that the weights are summable without any scaling. Such a phase transition at $\alpha=1/2$ of the cumulative weights is novel and unexpected as there is no significant change in the neighbourhood scaling of the IPC at $\alpha=1/2$.
    %We provide an intuitive explanation of the this particular phase transition at an unusual regime $\alpha=1/2$ and provide an explicit example where we determine the scaling of the expected cumulative weights.
\end{abstract}
%\tableofcontents

\section{Introduction}
In this paper we discuss the scaling of the cumulative weights added by \textit{invasion percolation} \cite{Wilkinson1983} on a weighted \textit{branching process tree} (BP tree) \cite[Chapter 8]{Karlin1975}.
Invasion percolation is a dynamic process, starting from the root, that creates a weighted sub-tree by every step adding the edge with the lowest weight incident to the the current sub-tree. By performing invasion percolation for an infinite number of steps, we obtain the sub-tree called the \textit{invasion percolation cluster} (IPC). A key feature of the IPC is the self-organised criticality, meaning that the limsup of the added weights converges to a fixed value without tuning specific parameters.
Such models find practical use in the physics literature where invasion percolation represents the (capillary) displacements of fluids or gasses in a porous medium \cite{Chandler1982, Laidlaw1993,Norris2014}.

The main interest of this paper is the scaling of the size of the IPC, which has been discussed before \cite{Addario-Berry2012,Angel2008, Gundlach2022} for different types of tree-like objects.
However, when discussing the \textit{size} of weighted graph objects, there are typically two points of view: On the one hand,
one can express size in number of vertices, also called the \textit{volume scaling}, for example via the number of vertices in the IPC at a certain distance from the root.
On the other hand, one can express size in weights, for example via the cumulative weights in the IPC at a certain distance from the root.
Based on the applications both have significant merit. Knowing the scaling of the number of vertices in a certain radius gives particular insight on how fluids displace through the medium, for example on how far and how wide the cluster is.
On the other hand, the assigned weights to edges can contain valuable information, such as distance, time or cost of traversing a certain edge. In such cases knowing the scaling of the cumulative weights adds an additional insight on the total time or cost of the IPC up to a certain level.
This becomes more relevant as recent studies find social-economic applications for the invasion percolation model \cite{Gabrielli2007,Gabrielli2009}. For a comprehensive overview of applications in invasion percolation, we refer to \cite[Chapter 16]{Sahimi2023}

In \cite{Gundlach2022}, the geometry and volume scaling of the IPC on the BP tree were discussed in detail. A noticeable feature of the IPC is that it is \textit{one-ended} \cite[Corollary 2.3]{Michelen2019}, meaning that it contains a unique path to infinity, also called the \textit{backbone}. In \cite{Gundlach2022}, the volume scaling of the \textit{$k$-cut IPC} was analysed, which denotes the cluster when the edge between the $k$-th and $(k+1)$-th backbone vertex is cut.
It turns out that the specific scaling differs
for different power-law offspring distributions of the branching process tree.
In this work, we
cover the natural extension of the aforementioned work to the scaling of the cumulative weights, see Theorem \ref{thm:main}, in the annealed setting. We also provide specific conditions on the offspring distribution such that the weights in the IPC are summable (without any scaling involved). 

\paragraph{Notation.}
We use the standard notation $\xrightarrow{d}$, $\xrightarrow{\prob}$ and $\xrightarrow{a.s.}$ to indicate convergence in distribution and probability and almost sure convergence respectively. 
Moreover, we denote $\xrightarrow{\D}$ for $J_1$ convergence in the Skorohod space of c\`adl\`ag paths on $D[\varepsilon,\infty)$ (c.f. \cite[Chapter 12-13]{Billingsley1999}).
We use the Bachmann-Landau notation $O(\cdot)$ and $o(\cdot)$ for asymptotic behaviour and write $o_\prob(1)$ for a random variable that vanishes in distribution arbitrarily slowly. Finally, we use $\preceq$ to indicate stochastic domination of random variables.

\paragraph{Organisation of this section.}
We discuss the details of our main contributions in Section \ref{sec:main_contr}. Here we also provide some important notation and properties of the IPC. Then, in Section \ref{sec:intuition}, we provide intuition on the main results. We then close this section with an overview of the contents of the rest of this paper.

\subsection{Main contributions}
\label{sec:main_contr}
Consider an i.i.d. branching process (BP) tree $\T=(V,E)$ with offspring distribution $X$, where $X\geq1$ with probability 1, so that survival of the tree is ensured. We assume a power-law offspring distribution defined by
\begin{equation}
\label{eq:heavy_tail}
    1-F(x) = c_{\sss X}x^{-\alpha}\big(1+o(1)\big),
\end{equation}
for $\alpha>0$ and normalisation constant $c_{\sss X}$.
We define the \textit{rooted weighted tree}
$(\T,\groot, \W)$ as the tuple containing the tree $\T$ with root vertex $\groot$, where weights are assigned to the edges of $\T$ represented by the weights sequence $\W$.
In the following we assume that realisations of $\W$  are i.i.d. from a standard uniform distribution. 
On this tree we define a sequence of subtrees $((T_n,\groot,\W(T_n))_{n\geq 0}$, where $T_n=(V_n,E_n)$ and $\W(T_n)$ is the natural restriction of weight sequence $\W$ to the edges in $E_n$.
This sequence represents a process called \textit{invasion percolation},
which we define iteratively. Start with rooted weighted tree $T_0=(\varnothing,\varnothing,\W(T_0))$, then $V(T_n)= V(T_{n-1})\cup \{v_n\}$, where $v_n$ is  defined by 
\begin{equation}
\label{eq:prim_steps}
    v_n = \text{argmin}_{u_n}\Big\{w(e_n): e_n=\{v_n,u_n\}, v_n\in V(\T)\setminus V(T_{n-1}), \exists u_n\in V(T_{n-1}), (v_n,u_n)\in E(T) \Big\},
\end{equation}
 and $E(T_n)=E(T_{n-1})\cup \{e_n\}$. 
 Intuitively, invasion percolation at every step adds the vertex via minimal edge weight from the current cluster. Invasion percolation shares this similarity \cite{Barabasi1996} with Prim's algorithm \cite{Prim1957}, an algorithm commonly used for finding the minimal spanning tree on finite weighted graphs. 
 
 We define the \textit{invasion percolation cluster} by the rooted weighted tree
\begin{equation}
    T=\lim_{n\to\infty} T_n.
\end{equation}
The weighted BP tree,
$(T,\varnothing, \W(T))$, has been extensively analysed in \cite{Gundlach2022}. The authors showed that $T$ consists of a unique path to infinity of vertices $(v_i)_{i\geq 0}$ with $v_0=\varnothing$. On this path, finite forests are attached. We refer to Figure \ref{fig:ipcgeo} for a schematic overview of a realisation of the IPC. A key part of the analysis regards the \textit{future maximum weight} sequence $(W_i)_{i\geq 0}$. This sequence denotes the largest weight added after the $i$-th backbone vertex. One also finds
\begin{equation}
\label{eq:convwk}
    \limsup_{n\to\infty} w(e_n) = \lim_{k\to\infty} W_k = p_c = \expec[X]^{-1},
\end{equation}
under the convention that $p_c=0$ if $\expec[X]=\infty$. We note that $p_c$ is also known as the critical percolation probability.
The main theorem from the aforementioned article regarded the volume scaling of the so-called $k$-cut IPC, denoted by $\M_k$, and defined by the object after cutting the edge between the $k$-th and $(k+1)$-th edge on the path to infinity, see \cite[Definition 1.1]{Gundlach2022}. 
In our work we extend this analysis to scaling of the \textit{weights} in $\M_k$. More specifically, we are interested in  
\begin{equation}
\label{eq:def_texttt_B_k}
    \B_k = \sum_{e\in E(\M_k)} w(e).
\end{equation}
The scaling of $M_k = |\M_k|$ is shown to have three different regimes: $\alpha>2$, $\alpha\in (1,2)$ and $\alpha\in(0,1)$, where between regimes the scaling of the $M_k$ varies significantly. Most interestingly, when discussing the scaling of the cumulative weights, this notion of regimes shifts.
For convenience, we now distinguish between three \textit{instances}.
We define the \textit{first instance} by $\alpha\in(1,\infty)\setminus \{2$\}, the \textit{second instance} by $\alpha\in (1/2,1)$ and the \textit{third instance} by $\alpha\in(0,1/2)$. With this we present our main result:

\begin{theorem}[Weight scaling of the IPC]
\label{thm:main}
Let $\B_k$ denote cumulative weights of the $k$-cut IPC (c.f. \eqref{eq:def_texttt_B_k}) and denote the number of vertices in $\M_k$ by $|\M_k|=M_k$. Then $\B_k$ satisfies the following scaling result as $k\to\infty$:

\noindent\textbf{First instance.} Assume $\alpha>2$ or $\alpha\in(1,2)$. Furthermore, suppose that there exists a $\gamma\geq 2$ and an a.s. finite random process $(Z_\alpha(t))_{t>0}$
such that $(k^{-\gamma}M_{\lceil kt\rceil})_{t>0}\xrightarrow{\D} (\int_0^t Z_\alpha(u)\dif u)_{t> 0}$. Then,
\begin{equation}
  \big(k^{-\gamma}  \B_{\lceil kt\rceil}\big)_{t> 0} \xrightarrow{\D}  \Big(\frac{p_c}{2}
  \int_0^t  Z_\alpha(u)\dif u \Big)_{t>0}
  .
\end{equation}

\noindent\textbf{Second instance.} Assume $\alpha \in (1/2,1)$. Furthermore, suppose that there exists an a.s. finite random variable $Z_\alpha$
such that $W_{k}^{\alpha/(1-\alpha)}(M_{ k} -M_{k-1})\xrightarrow{d} Z_\alpha$. Then there exists a sequence $(\tilde Z_\alpha(\ell))_{\ell\geq 1}$ (defined in \eqref{eq:def_tildeZ}) and an $m>0$, such that,
\begin{equation}
  \big( W_k^{\alpha/(1-\alpha)-1} \B_{k-\ell}\big)_{l=1}^m \xrightarrow{d} (\tilde Z_\alpha(\ell) )_{\ell=1}^m.
\end{equation}

\noindent\textbf{Third instance.}
Assume $\alpha\in(0,1/2)$. Then
there exists an a.s. finite random variable $\tilde Z_\alpha$ such that 
\begin{equation}
    \B_{k}  \xrightarrow{d} \tilde Z_\alpha.
\end{equation}
\end{theorem}
While the conditions imposed on $M_k$ may seem restrictive in the first and second instance, it is readily shown that they are satisfied as shown in \cite[Theorem 1.3]{Gundlach2022}. We also refer to this paper for the exact expressions of $Z_\alpha$ for $\alpha>2$, $\alpha\in(1,2)$ and $\alpha\in(1/2,1)$. In Section \ref{sec:intuition} we discuss the intuition behind Theorem \ref{thm:main}.

In the following we prove Theorem \ref{thm:main} per instance. As the scaling and the limiting random variables of $M_k$ are known, the limiting distribution and scaling for $\B_k$ are also explicit in the first instance. In the second instance, the scaling limit can also be determined in a straightforward manner. However, in the third instance we do not know $\tilde Z_\alpha$ explicitly, but we provide an example distribution for which we explicitly calculate $\expec[\tilde Z_\alpha(t)]$.

Deconstructing the IPC may seem daunting as, other than its construction, its geometry is not directly clear. However, it turns out that by conditioning on the future maximum weight sequence $(W_k)_{k\geq 0}$, the $k$-cut IPC can be seen as a sequence of $k$ finite independent forests, in which the weights are  i.i.d. uniformly distributed between 0 and $W_k$, c.f. \cite[Section 2]{Gundlach2022}. For $k\geq 0$, let $B_k = \B_k-\B_{k-1}$, i.e. the cumulative weights attached to backbone vertex $k$, under the convention that $\B_{-1}=0$, then: 

\begin{lemma}[Scaling of $B_k$]
\label{lem:scaling_B_k}
Consider the set of vertices $\H_k=\M_k\setminus\M_{k-1}$, i.e. all vertices that are in $T$ via the $k$-th backbone vertex and set $|\H_k|=H_k$. Condition on $W_k=w_k$, then the edge weights in $\H_k$, $(w(e_i))_{i=1}^{H_k}$, are i.i.d. $U_i\sim \Unif[0,w_k]$ distributed, i.e.,
\begin{equation}
\label{eq:def_B_k}
    B_k \equiv \sum_{i=1}^{H_k} U_i.
\end{equation}
\end{lemma}
Additionally, this also implies conditional independence of the attached forests of different backbone vertices:
\begin{corollary}[Independence of $B_k$]
\label{cor:indep_B_k}
    Consider $B_k$ as defined in \eqref{eq:def_B_k}. Then, conditionally on $(W_k)_{k\geq 0}=(w_k)_{k\geq 0}$, $(B_i)_{i\geq 1}$ is a collection of independent random variables.
\end{corollary}
Based on Lemma \ref{lem:scaling_B_k} and Corollary \ref{cor:indep_B_k}, it follows that, conditionally on $(W_k)_{k\geq 0} = (w_k)_{k\geq 0}$, 
\begin{equation}
    \B_k = \sum_{i=1}^k B_k
    =\sum_{i=1}^k\sum_{j=1}^{H_i} U_{ij},
\end{equation}
where $U_{ij}\sim\Unif[0,w_i]$, and the $(H_i)_{i=0}^k$ are independent. We refer to Figure \ref{fig:ipcgeo} for an overview of the aforementioned objects on the IPC.

\begin{figure}
    \centering
    \input{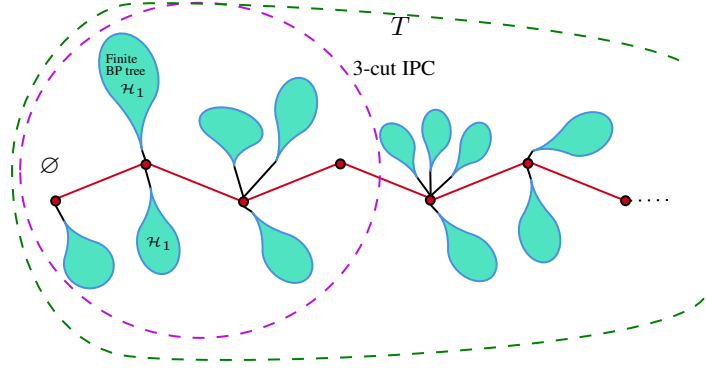}
    \caption{Schematic overview of the IPC. We highlight the backbone (red), the finite attached forests $\H_k$ (cyan) in the $k$-cut IPC (purple ball).}
    \label{fig:ipcgeo}
\end{figure}

\subsection{Intuition behind the main result}
\label{sec:intuition}
In the first instance,  
Theorem \ref{thm:main} outlines a clear and intuitive connection between the scaling of $\B_k$ and $M_k$. Indeed, the result is representative of a law-of-large-numbers result. This is not trivial, as the weights in the $k$- cut IPC are not i.i.d. 
It can be intuitively understood, however, by \cite[Theorem 1.3]{Gundlach2022}, where it is shown that the dominant contribution is from a few finite forests that are at least distance $k\varepsilon$ from the root. At such distance, $W_k$ is close to $p_c$ and, combining this with the results from Lemma \ref{lem:scaling_B_k} and Corollary \ref{cor:indep_B_k}, the discovered weights are nearly i.i.d. Therefore, the majority of the weights in the $k$-cut IPC is likely to contribute $\expec[\Unif[0,W_k]]\approx p_c/2$.

In the second and third instance, where $W_k\xrightarrow{a.s.} p_c=0$, the specific scaling changes. In contrast to the result in \cite[Theorem 1.3]{Gundlach2022}, where the dominant contribution is from large degrees of backbone vertices far way from the root, these vertices are likely to contribute a negligible weight. However, from the aforementioned source, we also find  $M_k$ scales as $ W_k^{-\alpha/(1-\alpha)}$ and, using the same argumentation as in the first instance, we would then find
\begin{equation}
    B_k \simeq W_k/2 \cdot W_k^{-\alpha/(1-\alpha)} \simeq
    W_k^{(1-2\alpha)/(1-\alpha)}.
\end{equation}
As $W_k\xrightarrow{a.s.} 0$ in these instances, this diverges only in the second instance, i.e. when $\alpha\in(1/2,1)$. Moreover, as $W_k$ converges exponentially fast (see \cite[Theorem 1.8]{Gundlach2022} or Example \ref{ex:expec cumulative weights}), in the third instance, $\B_k$ should converge without an additional scaling factor.

\paragraph{Organisation of the paper}
We start by providing some insights on the IPC in Section \ref{sec:detailsIPC}. These results were derived in detail in \cite[Section 2-3]{Gundlach2022} and are here repeated as they are key for the upcoming analysis. Moreover, the underlying construction of the IPC is used to prove Lemma \ref{lem:scaling_B_k} in Section \ref{sec:prooflemma1.2}. We then move to the proof of Theorem \ref{thm:main}. In Section \ref{sec:firstinstance}, we provide the proof for the first instance and in Section \ref{sec:secondinstance} for the second instance. Finally, in Section \ref{sec:thirdinstance} we provide the proof for the third instance and also provide a worked-out example in Section \ref{sec:worked out example}.

\section{Details on the IPC}
\label{sec:detailsIPC}
In the following, we give a short overview of how $T$ is constructed. We do not go in detail, but refer to \cite[Section 2-3]{Gundlach2022} for more results and proofs on the results presented below. 

\subsection{Properties of branching processes}
We write $\T(p)$ for the weighted connected component of an i.i.d. branching process that is \textit{percolated} with parameter $p$. In other words, $\T(p)$ is the connected component that includes the root after all edges with weight larger than $p$ are removed. $\T(p)$ is then again a branching process with offspring distribution $\Binom(X,p)$. We define $\eta(p)$ as its extinction probability and $\theta(p)=1-\eta(p)$ as its survival probability. From classical branching process theory, we can deduce that 
\begin{equation}
\label{eq:eta(p)}
    \eta(p) = 
    \expec\big[ \eta(p)^{\Binom(X,p)}\big]
    =\expec\big[(1-p\theta(p))^X\big].
\end{equation}
Moreover, as $X\geq 1$ and based on \eqref{eq:convwk}, we find $\theta(1)=1$ and $\theta(p_c)=0$.
In the upcoming analysis, we need the scaling of $\theta(p)$ around $p_c=1/\expec[X]$. This is given as follows:
\begin{lemma}[scaling of $\theta(p)$]
\label{lem:scaling_theta}
For $p\to p_c$
\begin{equation}
    \theta(p) = 
    \begin{cases}
         (p-p_c)(1+o(1)) & \alpha>2,\\
          (p-p_c)^{1/(\alpha-1)}(1+o(1))& \alpha\in(1,2),\\
          p^{\alpha/(1-\alpha)}(1+o(1))& \alpha\in(0,1).\\
    \end{cases}
    \end{equation}
\end{lemma}
\begin{proof}
    \cite[Lemma 3.1]{Gundlach2022}.
\end{proof}
\subsection{Construction of the IPC}
Firstly, as described before, $T$ is \textit{one-ended}, meaning it has a unique path to infinity $(v_i)_{i\geq0}$ and denote the corresponding weights on this path by $(\beta_i)_{i\geq1}$
\cite[Lemma 2.1]{Gundlach2022}
. We define the future maximum weights by
\begin{equation}
    W_k = \max_{i>k} \beta_i.
\end{equation}
It is also known that $W_k$ evolves according to a Markov chain given by 
\begin{equation}
\label{eq:Wk_Mc}
    \begin{aligned}
        \prob(W_{k+1}
        =W_k\mid W_k) &= 1-\expec[X(1-p\theta(p))^{X-1}]
        \frac{\theta(W_k)}{\theta'(W_k)},\\
        \prob(W_{k+1}
        <u\mid W_k) &= \expec[X(1-p\theta(p))^{X-1}]
        \frac{\theta(u)}{\theta'(W_k)},
        \qquad u\in[0,W_k].
    \end{aligned}
\end{equation}
See \cite[Theorem 1.6]{Gundlach2022} for a detailed overview.
It turns out that conditioning on $W_k=w_k$ is paramount in the upcoming analysis and therefore we will write
\begin{equation}
\label{eq:prob_k}
    \prob_k(\cdot) = \prob(\ \cdot\mid W_k=w_k)
   \quad \text{ and }\quad  
    \expec_k[\cdot ] = \expec_k[\ \cdot\mid W_k=w_k].
\end{equation}
Conditioned on $\{W_k=w_k\}$, the distribution of the $k$-th backbone degree $D_k$ is given by 
\begin{equation}
\label{eq:prereq_backbonedegree}
    \prob_k(D_{v_k}= x) = 
    \frac{x(1-w_k\theta(w_k))^{x-1}\prob(X=x)}{\expec_k[X(1-w_k\theta(w_k))^{X-1}]},
\end{equation}
On each of the $D_{v_k}$ edges of the $k$-th backbone vertex, a conditionally finite tree is attached. We write $\tilde X_k$ for its offspring distribution and find
\begin{equation}
\label{eq:prereq_tildeX}
    \prob_k(\tilde X_k= x) = 
    \eta(w_k)^{x-1}\sum_{y>x}\binom{y}{x} w_k^x(1-w_k)^{y-x} \prob(X=y).
\end{equation}
With these ingredient, one can systematically construct the IPC. For more details on this construction algorithm or for more detail on \eqref{eq:prereq_backbonedegree} and \eqref{eq:prereq_tildeX}, we refer to \cite[Theorem 1.5]{Gundlach2022}.

\subsection{Weights on the IPC}
\label{sec:prooflemma1.2}
The weights on the IPC can be split in two classes, those on the backbone of the IPC and those on the finite attached forests. First, as weights are standard uniformly distributed, the weights on the backbone of the IPC are trivially bounded by  
 \begin{equation}
 \label{eq:bound_weights_backbone}
     \sum_{i=1}^k \beta_i
     \leq k.
 \end{equation}
Second, the weights in the finite attached clusters are more crucial to our analysis and are formalised in Lemma \ref{lem:scaling_B_k}. We  provide the proof for the result below:

 \begin{proof}[Proof of Lemma \ref{lem:scaling_B_k}]
  Consider $u_k\in \T$, such that $v_k$ is the closest backbone vertex (in graph distance) to $u_k$ and condition on $W_k=w_k$. We then claim that $u_k\in T$ if and only if the path between $u_k$ and $v_k$ consists of weights only below $w_k$

  \textbf{If.} Suppose that the path between $u_k$ and $v_k$ has weights only below $w_k$. Recall that, by conditioning on $W_k=w_k$, at some point a weight of exactly $w_k$ is added by invasion percolation. However, based on 
  \eqref{eq:prim_steps}, it means that before that time, all weights below $w_k$ that are of distance 1 from the current invasion percolation are added first. If the path between $u_k$ and $v_k$ only consists of weights below $w_k$, $u_k$ is explored before the weight $w_k$ is added.

  \textbf{Only if.} Suppose there exists at least on weight $w^*>w_k$ on the path between $u_k$ and $v_k$. Recall that, by conditioning on $W_k=w_k$ that
  the weight $w_k$ is the largest weight that invasion percolation will add. This means that $w^*$ is not explored by invasion percolation, and, as a result, neither is $u_k$.

  Then, we take $e_1$ and $e_2$ edges in $H_k$ and condition on $W_k=w_k$. Based on the observation above, the weights of $e_1$ and $e_2$ must be smaller than $w_k$. As weights are assigned i.i.d. standard uniformly on all edges, we find
  \begin{equation}
  \begin{aligned}
      \prob_k( w(e_1)<x_1, w(e_2)<x_2) &=
      \prob_k(w(e_1)<x_1, w(e_2)<x_2 \mid 
      w(e_1)<w_k, w(e_2)<w_k)\\
      &=
      \prob_k(w(e_1)<x_1\mid  
      w(e_1)<w_k) \prob_k( w(e_2)<x_2 \mid w(e_2)<w_k)\\
      &=
      \prob_k(\Unif[0,w_k]<x_1)
      \prob_k(\Unif[0,w_k]<x_2).
  \end{aligned}    
  \end{equation}
 \end{proof}

\section{Weight scaling in the first instance}
\label{sec:firstinstance}
In this section we give the proof of the first instance of Theorem \ref{thm:main}. In this instance the supremum of the weights that are included over time stay significantly large as illustrated in \eqref{eq:convwk}, which explains the result for the first instance.
The proof is split in three parts that we outline below in more detail:

\begin{enumerate}
\item \hyperref[inst1_step1]{Part 1: Convergence of the one-dimensional conditional transform.}
The proof uses the Laplace transform of $\B_k$ by relating it to the transform of $M_k$ in a straightforward manner.
Assume $(w_k)_{k\geq 0}$ is a non-increasing process between $[0,1]$ that converges to $p_c$. Then we prove the first instance conditioned on $(W_k)_{k\geq 0} = (w_k)_{k\geq 0}$.
\item\hyperref[inst1_step2]{Part 2: Extension to the one-dimensional unconditional transform.}
Part 1 shows functional convergence of the transform in terms of $(w_k)_{k\geq0}$. As we know that $W_k\xrightarrow{a.s} p_c$, we apply the continuous function theorem in a strategic way to extend the result of part 1 to convergence in distribution after which we conclude the result by applying the Portmanteau theorem.
\item\hyperref[inst1_step3]{Part 3: Extension to $J_1$ convergence.}
When considering $\B_{\lceil kt\rceil}$ as a
random functional in $t$, Part 2 shows, in essence, pointwise convergence. In this part we extend this result to $J_1$ convergence, which can be interpreted as convergence of the full process. In order to extend the result, it is sufficient to show convergence of the finite dimensional distribution and tightness.
\end{enumerate}

\begin{proof}[Proof of Theorem \ref{thm:main}, first instance]
Assume there exists a $\gamma\geq 2$ such that $(k^{-\gamma}M_{\lceil kt\rceil})_{t>0}
\xrightarrow{\mathcal{D}} (\int_0^t Z_\alpha(u)\dif u)_{t>0}$.\\
\noindent\textbf{Part 1: Convergence of the conditional transform}
\label{inst1_step1}
Recall $B_k$ from \eqref{eq:def_B_k} and recall the notation of \eqref{eq:prob_k}, then for some non-increasing sequence $(w_k)_{k\geq 0}$ such that $w_k\to p_c$ and for any $s>0$ and fixed $t>0$, we find by Lemma \ref{lem:scaling_B_k},
    \begin{equation}
    \label{eq:transform_Bk_firstregime_1}
        \expec_k[\e^{-sk^{-\gamma}B_k}]
        =\expec_k\Big[\e^{-sk^{-\gamma}\sum_{e\in E(\H_k)} w(e)}\Big]
        =\expec_k\Big[\expec_k[\e^{-sk^{-\gamma}\Unif[0,w_k]}]^{H_k}\Big].
    \end{equation}
Next, we can write for the transform of a uniform random variable
\begin{equation}
\label{eq:transform_uniform}
    \expec\big[\e^{-s\Unif[0,1]}\big]
    =\frac{1}{s}(1-\e^{-s}) =
    \e^{\log\big(1-s/2+s^2/6-O(s^3)\big)} =
    \e^{-s/2+O(s^2)}.
\end{equation}
As $\Unif[0,w_k] \equiv w_k\Unif[0,1]$ and based on \eqref{eq:transform_uniform} we find that \eqref{eq:transform_Bk_firstregime_1} simplifies to
\begin{equation}
    \label{eq:transform_Bk_firstregime_2}
    \expec_k\big[\e^{-sk^{-\gamma}B_k}\big]
        =
    \expec_k\Big[\e^{-\frac{1}{2}sw_kk^{-\gamma}H_k+ O(k^{-2\gamma}H_k)}\Big ].
\end{equation}
This result is sufficient to derive the scaling of $\B_k$. Based on Corollary \ref{cor:indep_B_k}, we split $\B_k$ in contributions from the backbone and from the finite attached clusters $B_k$, and find by \eqref{eq:bound_weights_backbone} that for $\gamma\geq 2$
\begin{equation}
\label{eq:transform_Bk_firstregime_3}
    \expec_k\big[\e^{-sk^{-\gamma}\B_{\lceil kt\rceil}}\big]
    =\expec_k\Big[\e^{-sk^{-\gamma}\sum_{i=1}^{\lceil kt\rceil} \beta_i
     -sk^{-\gamma}\sum_{i=1}^{\lceil kt\rceil} B_i}\Big]
     =\expec_k\Big[\e^{
     -sk^{-\gamma}\sum_{i=1}^{\lceil kt\rceil} B_i
     +O(k^{_1-\gamma })}\Big].
\end{equation}
As the $\H_i$ clusters are independent conditioned on $(W_k)_{k\geq0}=(w_k)_{k\geq 0}$, we can write 
\begin{equation}
\label{eq:transform_Bk_firstregime_4}
\begin{aligned}
    \expec_k\big[\e^{
     -sk^{-\gamma}\sum_{i=1}^{\lceil kt\rceil} B_{i}
     +O(k^{1-\gamma})}\big]
     &=
     \prod_{i=1}^{\lceil kt\rceil} \expec_k\big[\e^{
     -sk^{-\gamma} B_i +O(k^{1-\gamma})}\big]\\&=
      \prod_{i=1}^{\lceil kt\rceil}\expec_k\big[\e^{
    - \frac{1}{2}sw_kk^{-\gamma}H_i+ O(k^{1-\gamma})}\big]
    \\&= \expec_k\big[\e^{
     -\frac{1}{2}sp_ck^{-\gamma} M_{\lceil kt\rceil}}\big](1+o(1)).
\end{aligned}
\end{equation}
Finally we let $\varepsilon>0$ be given. We conclude that based on \eqref{eq:transform_Bk_firstregime_4}, there exists a $K_1$ large enough so that for all $k>K$,
\begin{equation}
    \Big|\expec_k\big[\e^{-sk^{-\gamma}\B_{\lceil kt\rceil}}\big]
    -\expec_k\big[\e^{-\frac{1}{2}sp_ck^{-\gamma}M_{\lceil kt\rceil}}\big] \Big|
    <\varepsilon/2,
\end{equation}
and based on the assumption that $k^{-\gamma}M_{\lceil kt\rceil} \xrightarrow{d} \int_0^t Z_\alpha(u)\dif u$, there exists a $K_2$ such that for all $k>K_2$ and all $s>0$
\begin{equation}
    \Big|\expec_k\big[\e^{-sk^{-\gamma}M_{\lceil kt\rceil}}\big]
    -\expec\big[\e^{-s\int_0^t Z_\alpha(u)\dif u }\big]\Big|
    <\varepsilon/2.
\end{equation}
Then, for all $k>\max\{K_1,K_2\}$,
\begin{equation}
\label{eq:conditiona conv alpha>1}
\begin{aligned}
    &\Big|\expec_k\big[\e^{-sk^{-\gamma}\B_{\lceil kt\rceil}}\big] - \expec\big[\e^{-\frac{1}{2}sp_c\int_0^t Z_\alpha(u)\dif u }\big]\Big|
    %\\&=\Big|\expec_k\big[\e^{-sk^{-\gamma}\B_{\lceil kt\rceil}}\big] - \expec_k\big[\e^{-\frac{1}{2}sp_ck^{-\gamma}M_{\lceil kt\rceil}}\big]\Big|+
   %  \Big|\expec_k\big[\e^{-\frac{1}{2}sp_ck^{-\gamma}M_{\lceil kt\rceil}}\big]
    %-\expec\big[\e^{-\frac{1}{2}sp_c\int_0^t Z_\alpha(u)\dif u }\big]\Big| \\&
    <\varepsilon/2 +\varepsilon/2=\varepsilon.
\end{aligned}  
\end{equation}
We conclude that $\expec_k[\e^{-sk^{-\gamma}\B_{\lceil kt\rceil}}]\to \expec[\e^{-\frac{1}{2}sp_c\int_0^t Z_\alpha(u)\dif u }]$. \\

\noindent\textbf{Part 2: Extension to the unconditional transform.}
\label{inst1_step2}
As \eqref{eq:conditiona conv alpha>1} holds for the class of non-increasing  sequences $(w_k)_{k\geq0}$ such that $w_k\in[p_c,1]$ for any $k$ and $w_k\xrightarrow{d} p_c$ and as $(W_k)_{k \geq 0}$ is with probability one in this class of sequences, we find by the generalised continuous function theorem \cite[Theorem 3.4.4]{Whitt2002}, implies 
\begin{equation}
    \expec\big[\e^{-sk^{-\gamma}\B_{\lceil kt\rceil} }\mid (W_k)_{k\geq0}\big]\xrightarrow{d} \expec\big[\e^{-\frac{1}{2}sp_c\int_0^tZ_\alpha(u)\dif u}\big].
\end{equation}
As both expectations are uniformly bounded by 1, the Portmanteau theorem extends this convergence to the unconditional expectation:  $\expec[\e^{-sk^{-\gamma}\B_{\lceil kt\rceil}}]\to \expec[\e^{-\frac{1}{2}sp_c\int_0^tZ_\alpha(u)\dif u}]$,
which directly implies the result.\\

\noindent\textbf{Part 3: Extension to $J_1$ convergence.}
\label{inst1_step3}
In the following we show convergence of the finite dimensional distribution and tightness of $\B_{\lceil kt\rceil}$. We conclude the result by showing that these properties imply $J_1$ convergence.

Fix $m\in \mathbb{N}^+$ and take $t_1<\ldots<t_m$ and $s_1,\ldots,s_m\in\mathbb{R}^+$.
Consider then the joint Laplace transform
\begin{equation}
    \expec_k\Big[
    \e^{-k^{-\gamma}\sum_{i=1}^m s_i\B_{\lceil kt_i\rceil}}
    \Big].
\end{equation}
As $\B_k$ is a sum of random variables (c.f. \eqref{eq:def_texttt_B_k}), it can be deconstructed, by Corollary \ref{cor:indep_B_k} in mutually independent parts
\begin{equation}
     \expec_k\Big[
    \e^{-k^{-\gamma}
    (s_1+\ldots+s_m)\sum_{i=0}^{\lceil kt_1\rceil}
    B_i+\ldots+
    k^{-\gamma} s_m \sum_{i=\lceil kt_{m-1}\rceil+1}^{\lceil kt_m\rceil} B_i
    }\Big]
    =
    \prod_{i=1}^m
    \expec_k\Big[
    \e^{-k^{-\gamma}
    \sum_{j=i}^{m} s_j\sum_{l=\lceil kt_{i-1}\rceil+1}^{\lceil kt_i\rceil}
    B_l
    }\Big],
\end{equation}
under the convention that $kt_0+1=0$. Using a similar trick to \eqref{eq:transform_Bk_firstregime_1}-\eqref{eq:transform_Bk_firstregime_2}, this can be rewritten to
\begin{equation}
    \prod_{i=1}^m
    \expec_k\Big[
    \e^{-k^{-\gamma}
    \sum_{j=i}^m s_j\sum_{l=\lceil kt_{i-1}\rceil+1}^{\lceil kt_i\rceil}
    \sum_{e\in E(\H_k)}
    w(e)
    }\Big]
    =
    \prod_{i=1}^m
    \expec_k\Big[
    \e^{-k^{-\gamma}
    \sum_{j=i}^m s_j\sum_{l=\lceil kt_{i-1}\rceil+1}^{\lceil kt_i\rceil}
    H_i w_i/2
    }\Big]
\end{equation}
Then by the result of Part 1, these transforms individually converge and we find, 
\begin{equation}
\label{eq:transform_Bk_firstregime_fin}
\begin{aligned}
    \prod_{i=1}^m
    \expec_k\Big[
    \e^{-k^{-\gamma}
    \sum_{j=i}^m s_j\sum_{l=\lceil kt_{i-1}\rceil+1}^{\lceil kt_i\rceil}
    H_i w_i/2
    }\Big] &\to
    \prod_{i=1}^m
    \expec_k\Big[
    \e^{-p_c/2
    \sum_{j=i}^m s_j\int_{t_{i-1}}^{t_i}
    Z_\alpha(u)\dif u 
    }\Big]
    \\&=
    \expec_k\Big[
    \e^{-p_c/2(s_1\int_0^{t_1} Z_\alpha(u)\dif u+\ldots + s_m\int_{t_{m-1}}^{t_m} Z_\alpha(u)\dif u)}
    \Big].
\end{aligned}
\end{equation}
Finally, one can apply a similar strategy as in Part 2, that extend this result to the unconditional finite-dimensional distribution.

We show tightness next. The main idea is based on the observation
\begin{equation}
\label{eq:stochdomB}
    k^{-\gamma}\B_{\lceil kt\rceil}
    =
    k^{-\gamma}\sum_{i=1}^{\lceil kt\rceil}
    B_i
    \preceq 
     k^{-\gamma}\sum_{i=1}^{\lceil kt\rceil}
    H_i
    =k^{-\gamma}M_{\lceil kt\rceil}.
\end{equation}
As $k^{-\gamma}M_{\lceil kt\rceil}$ is tight by \cite[Appendix B2]{Gundlach2022}, this implies tightness of $ k^{-\gamma}\B_{\lceil kt\rceil}$ in a straightforward manner. For the details we refer to Appendix \ref{app:tightness}.
Finally, based on \cite[Theorem 13.1]{Billingsley1999}, convergence of the finite dimensional distribution, as shown in \eqref{eq:transform_Bk_firstregime_fin}, and tightess, as shown in Appendix \ref{app:tightness}, implies $J_1$ convergence.
\end{proof}

\section{Weight scaling in the second instance}
\label{sec:secondinstance}
In this section we assume $\alpha\in(1/2,1)$. We present a key result on the $(W_k)_{k\geq 0}$ process that is relevant in the upcoming scaling results:

\begin{lemma}[Scaling of the ratio of $W_k$]
\label{lem:ratio_Wk}
    Fix the i.i.d. sequence $(P_i)_{i\geq1}$, where 
    \begin{equation}
        P_1 = 
        \begin{cases}
            1 & \text{w.p. } \alpha,\\
            \Unif[0,1]^{(1-\alpha)/\alpha}
            & \text{w.p. } 1-\alpha.
        \end{cases}
    \end{equation}
    Then, for fixed $\ell>0$ and $k\to\infty$,
    \begin{equation}
        \bigg(
        \frac{W_k}{W_{k-i}}
        \bigg)_{i=1}^{\ell}
        \xrightarrow{d}
        \bigg(
        \prod_{j=1}^i P_j
        \bigg)_{i=1}^\ell.
    \end{equation}
\end{lemma}
\begin{proof}
    \cite[Lemma 3.15]{Gundlach2022}.
\end{proof}
This implies that, asymptotically, one can observe $W_k$ as an i.i.d. product of random variables $P_i$.
Next, we provide the proof of Theorem \ref{thm:main} for the second instance. We use a similar strategy as posed in the first instance, that we briefly outline below.

\begin{enumerate}
\item \hyperref[inst2_step1]{Part 1: Convergence of the one-dimensional conditional transform.}
We again show convergence via the Laplace transform. Assume that $(w_k)_{k\geq 0}$ is a non-increasing process between $[0,1]$  that converges to $p_c$ and additionally
that $w_k/w_{k-\ell}\to p(\ell)\in(0,1)$ for any $\ell$.
Then we prove the first instance conditioned on $(W_k)_{k\geq 0} = (w_k)_{k\geq 0}$.
\item \hyperref[inst2_step2]{Part 2: Extension to the one-dimensional unconditional transform.}
Based on Lemma \ref{lem:ratio_Wk}, $(W_k)_{k\geq 0}$ satisfies the condition set on $(w_k)_{k\geq 0}$.
We can therefore extend the functional convergence of the transform in terms of $(w_k)_{k\geq0}$ based on the continuous function theorem.

\item \hyperref[inst2_step3]{Part 3: Extension to joint convergence.}
We finally extend Part 2 to convergence of the finite dimensional distribution. Due to the specific nature of the scaling of $\B_k$, additionally showing a process limit (for example for $\B_{\lceil kt\rceil}$) is less insightful. Indeed, it can be verified, see for example Remark \ref{rem:J_1_second_inst}, that the only candidate limiting process is independent of $t$.
\end{enumerate}

\begin{proof}[Proof of Theorem \ref{thm:main}, second instance]
We provide details to the three parts as described as above in order to prove Theorem \ref{thm:main} for the second instance. In the following, assume $\alpha\in(1/2,1)$ and that
$W_{k-\ell}^{\alpha/(1-\alpha)}H_{k-\ell}\to Z_\alpha(\ell)$, an i.i.d. copy of $Z_\alpha$.\\

\noindent\textbf{Part 1: Convergence of the one-dimensional conditional transform.}
\label{inst2_step1}
   Fix $\ell\in\mathbb{N}$. 
   Recall $B_k$ from \eqref{eq:def_B_k} and the notation of \eqref{eq:prob_k}, then assume $(w_k)_{k\geq 0}$ a non-increasing sequence such that, for $k\to\infty$, $w_k\to 0$, and $w_k/w_{k-\ell}\to p(\ell)\in(0,1)$ for any $\ell$. Then we start with the conditional transform. Based on a similar approach to \eqref{eq:transform_Bk_firstregime_1}- \eqref{eq:transform_Bk_firstregime_2}, we find for $s>0$
   \begin{equation}
\label{eq:transform_Bk_secregime_1}
   \begin{aligned}
       \expec_k\left[\e^{-sw_k^{\alpha/(1-\alpha)-1}\B_{k-\ell} }\right]
      & = \expec_k\left[ 
       \e^{-sw_k^{\alpha/(1-\alpha)-1}
       \sum_{i=0}^{k-\ell}
       \sum_{j=1}^{H_{i}}
       \Unif_j[0,w_i]
       (1+o(1))
       }
       \right]\\
        & = \expec_k\left[ 
       \e^{-\frac{1}{2}sw_k^{\alpha/(1-\alpha)-1}
       \sum_{i=\ell}^{k} w_{k-i}
        H_{k-i}
       (1+o(1))
       }
       \right]
       \\&=\expec_k\bigg[ 
       \e^{-\frac{1}{2}s\sum_{i=\ell}^k\frac{w_k^{(2-\alpha)/(1-\alpha)}}{w_{k-i}^{(2\alpha-1)/(1-\alpha)}}
       w_{k-i}^{\alpha/(1-\alpha)}H_{k-i}(1+o(1))
       }
       \bigg].
    \end{aligned}
   \end{equation}
Based on the conditions set on $(w_k)_{k\geq 0 }$, we find for $k\to\infty$
\begin{equation}
\label{eq:conv_weight_reg_2_1}
    \expec_k\bigg[ 
       \e^{-\frac{1}{2} s\sum_{i=\ell}^k\frac{w_k^{(2-\alpha)/(1-\alpha)}}{w_{k-i}^{(2-\alpha)/(1-\alpha)}}
       w_{k-i}^{\alpha/(1-\alpha)}H_{k-i}(1+o(1))
       }
       \bigg]
    \to   \expec_k\left[ 
       \e^{-\frac{1}{2}s\sum_{i=\ell}^\infty
       \prod_{j=1}^i p(j)^{(2\alpha-1)/(1-\alpha)}Z_\alpha(i)}
       \right].
\end{equation}
\noindent\textbf{Part 2: Extension to the one-dimensional unconditional transform.}
\label{inst2_step2}
As we find by Lemma \ref{lem:ratio_Wk}, we find that $(W_k)_{k\geq 0}$ is with probability 1 in the class of sequences $(w_k)_{k\geq 0}$ that we assumed for \eqref{eq:conv_weight_reg_2_1}. Therefore, by the continuous function theorem, we also find for $(P_i)_{i\geq 0}$ as defined in 
Lemma \ref{lem:ratio_Wk},
\begin{equation}
\label{eq:conv_weight_reg_2_2}
   \expec\Big[ 
       \e^{-sW_k^{(2\alpha-1)/(1-\alpha)} \B_{k-\ell}
       }\ \Big | \ (W_i)_{i=0}^k
       \Big]
    \xrightarrow{d}   \expec\Big[ 
       \e^{-\frac{1}{2}s
       \sum_{i=\ell}^\infty
       \prod_{j=1}^i P_j^{(2\alpha-1)/(1-\alpha)}Z_\alpha(i)}
       \ \Big | \ (W_i)_{i=0}^k
       \Big].
\end{equation}
   Since both expectations are uniformly bounded by 1, convergence of the unconditional Laplace transform follows from the Portmanteau theorem and we conclude 
\begin{equation}
\label{eq:conv_weight_reg_2_3}
   \expec\Big[ 
       \e^{-sW_k^{(2\alpha-1)/(1-\alpha)} \B_{k-\ell}
       }
       \Big]
    \xrightarrow{d}   \expec\Big[ 
       \e^{-\frac{1}{2}s
       \sum_{i=\ell}^\infty
       \prod_{j=1}^i P_j^{(2\alpha-1)/(1-\alpha)}Z_\alpha(i)}
       \Big].
\end{equation}

\noindent\textbf{Part 3: Extension to joint convergence.} 
\label{inst2_step3}
Fix $m>0$, integers $0<\ell_1<\ldots <\ell_m$ and $s_1,\ldots s_m\in\mathbb{R}^+$. We then find for the joint transform by rearranging terms that
\begin{equation}
\label{eq:jointconv_inst2_1}
\begin{aligned}
    &\expec_k\Big[ 
    \e^{-w_{k}^{\alpha/(1-\alpha)-1}(
    s_1 \sum^{k-\ell_1}_{i=0}B_i+\ldots +
    s_m \sum^{k-\ell_m}_{i=0}B_i
    )
    }
    \Big]
    \\&\qquad=
     \expec_k\Big[ 
    \e^{-w_{k}^{\alpha/(1-\alpha)-1}(
    (s_1+\ldots+s_m) \sum^{k-\ell_m}_{i=0}B_i+\ldots +
    s_1
    \sum^{k-\ell_1}_{i=k-\ell_2+1}B_i)}
    \Big]
    \\&\qquad=
    \prod_{i=1}^m \expec_k\Big[ 
    \e^{-w_{k}^{\alpha/(1-\alpha)-1}
    \sum_{j=i}^m s_{m-j+1}
    \sum_{j=\ell_{m-i+1}+1}^{\ell_{m-i+2}}B_{k-j}}
    \Big],
\end{aligned}
\end{equation}
under the convention that $\ell_{m+1}=k$.
Based on \eqref{eq:conv_weight_reg_2_3} taking the limit gives
\begin{equation}
\eqref{eq:jointconv_inst2_1}
    \to 
     \prod_{i=1}^m \expec_k\Big[ 
    \e^{-\frac{1}{2}p_c\sum_{j=i}^m s_{m-j+1}
    \sum_{j=\ell_{m-i+1}+1}^{\ell_{m-i+2}}
    \prod_{u=1}^jp(u)^{ (2\alpha-1)/(1-\alpha)}Z_\alpha(j)
    }
    \Big].
\end{equation}
These terms can again be properly rearranged so that this expression simplifies to
\begin{equation}
   \expec_k \Big[ 
    \e^{-\frac{1}{2}p_cs_1
    \sum_{j=\ell_1}^{\infty}
    \prod_{u=1}^jp(u)^{ (2\alpha-1)/(1-\alpha)}Z_\alpha(j)
    +\ldots +
    s_m
    \sum_{j=\ell_m}^{\infty}
    \prod_{u=1}^jp(u)^{ (2\alpha-1)/(1-\alpha)}Z_\alpha(j)
    }
    \Big].
\end{equation}
Next, set 
\begin{equation}
\label{eq:def_tildeZ}
    \tilde Z_\alpha(\ell) = 
    \sum_{j=\ell}^{\infty}
    \prod_{u=1}^jP_j^{ (2\alpha-1)/(1-\alpha)}Z_\alpha(j),
\end{equation}
then by a similar approach as in Step 2, this result can be extended to the unconditional Laplace transform. We then find
\begin{equation}
    \expec\Big[ 
    \e^{-W_{k}^{\alpha/(1-\alpha)-1}(
    s_1\B_{k-\ell_1}+\ldots +
    s_m \B_{k-\ell_m}
    )
    }
    \Big]
    \to
    \expec\Big[ 
    \e^{-W_{k}^{\alpha/(1-\alpha)-1}(
    s_1\tilde Z_\alpha(\ell_1)+\ldots +
    s_m \tilde Z_\alpha(\ell_m) )
    }
    \Big],
\end{equation}
thereby concluding the result.
\end{proof}

\begin{remark}[Trivial $J_1$ convergence for the second instance]
\label{rem:J_1_second_inst}
Extending this result to process convergence of $(W_{k}^{\alpha/(1-\alpha)-1}\B_{\lceil kt \rceil})_{t>0}$ would result in a, less interesting, i.i.d. process. Indeed, by taking $\ell=0$ one can verify in a straightforward manner from \eqref{eq:transform_Bk_secregime_1} that for any fixed $t$ 
\begin{equation}
     W_{k}^{\alpha/(1-\alpha)-1}\B_{\lceil kt \rceil}
     \to 
     \tilde Z_\alpha(0).
\end{equation}
\end{remark}

\begin{comment}
\begin{remark}[$\alpha=1/2$]
In the specific case that $\alpha=1/2$, we observe that by the analysis above, 
    \begin{equation}
\label{eq:conv_weight_reg_2_3}
   \expec\Big[ 
       \e^{-s\B_{k-\ell}
       }
       \Big]
    \xrightarrow{d}   \expec\Big[ 
       \e^{-\frac{1}{2}s
       \sum_{i=\ell}^\infty
      Z_\alpha(i)}
       \Big].
\end{equation}
Therefore, in the case $\alpha=1/2$, the weights are summable and the  limiting distribution is explicit. 
\end{remark}
\end{comment}

\section{Weight scaling in the third instance}
\label{sec:thirdinstance}
We finally consider the third instance by taking $\alpha\in(0,1/2)$. Here, we find that the degrees of the vertices in the IPC are so large, that the weights discovered in the IPC over time get so small that they become summable. We first prove that $\B_k$ has bounded expectation in $k$, which implies that it is almost surely finite. Moreover, as $\B_k$ is increasing, it converges almost surely to an almost surely finite random variable. Then we illustrate for a specific example what the expected added weight would be.

\begin{proof}[Proof of Theorem \ref{thm:main}, third instance]
We show that $\B_k$ is almost surely finite by showing that $\expec[\B_k]<\infty$, so that the result follows from the Markov inequality. 
We start by taking a bound on $(W_k)_{k\geq 0}$. According to \cite[Equation (3.68)]{Gundlach2022} there exists constants $\bar a> \underline a>0$ such that $\underline a >\alpha/(1-\alpha) \bar a$ and $K>0$ such that for all $k>K$
\begin{equation}
\label{eq:boundsW_k_reg3}
   \e^{-\bar{a}k}< W_k < \e^{-\underline a k},
\end{equation}
with high probability. 
%Define a sequence $\tilde w_k = \1\{k<K\} + \e^{-\underline ak}\1\{k\geq K\}$ to ensure $W_k\leq w_k$ with high probability.
We write $\hat D_{v_k}\equiv \Binom(D_{v_k},W_k)$ for the degree of $v_k$ restricted to vertices in the IPC 
and $T_i^{W_k}$ as the $i$-th conditionally finite tree in the conditionally finite forest attached to the $k$-th backbone vertex. Then, we find for the total cumulative weights attached to the $k$-the backbone vertex
\begin{equation}
\label{eq:third_rigime_eq0}
\begin{aligned}
    \expec_k[B_{k}\mid W_k]&=
    \expec_k\Big[
    \beta_k + \sum^{\hat D_{W_k}}_{i=1}\sum^{|T_i^{W_k}|+1}_{j=1} \Unif[0,W_k]\
    \Big|\ W_k\Big].
\end{aligned}
\end{equation}
Conditional on the future maximum weight $W_k$, the backbone degrees $\hat D_{W_k}$, finite tree sizes $|T_i^{(W_k)}|$, and the weights of edges inside the finite trees are independent by Lemma \ref{lem:scaling_B_k} and Corollary \ref{cor:indep_B_k}. It can therefore be written as 
\begin{equation}
\begin{aligned}
\label{eq:expec_B_k}
    \expec_k[B_{k}\mid W_k]
    = 
    \expec[\beta_k\mid W_k]+
    \frac{1}{2} \expec_k[\hat D_{W_k}\mid W_k ]\expec_k[|T_i^{W_k}|+1\mid W_k]
     W_k,
\end{aligned}
\end{equation}
where we recall that $\beta_k$ is the weight of the $k$-th edge on the backbone, and $T_i^{W_k}$ a tree that is restricted to weights below $W_k$ and is conditionally finite.
In the following we derive the three unknown expectations explicitly.\\

\noindent\textbf{Expected weight on the $k$-th backbone edge.}
Based on \eqref{eq:Wk_Mc}, $\beta_k$ is either equal to $W_k$, i.e., the next future maximal weight is added, or it is smaller than $W_k$, in which case it is simply uniformly between $[0,W_k]$. Therefore, 
\begin{equation}
    \expec[\beta_k\mid W_k ]
   \leq 
  W_k.
\end{equation}

\noindent\textbf{Expected weights on the $k$-the backbone vertex.}
We use our observation from \eqref{eq:prereq_backbonedegree} to find
\begin{equation}
\label{eq:bbweights_reg3}
    \expec[\hat D_{W_k}\mid W_k]=\frac{1}{2}W_k\expec[D_{v_k}]
=\frac{1}{2}W_k\frac{\expec[X^2(1-W_k\theta(W_k))^{X-1}\mid W_k]}{\expec[X(1-W_k\theta(W_k))^{X-1}\mid W_k]}.
\end{equation}
Next, based on the Karamata-Tauberian theorem \cite[Theorem 8.1.6]{Bingham1987}, we find the scaling for the expectation, for $k\to\infty$,
\begin{equation}
\begin{aligned}
    \expec[X(1-W_k\theta(W_k))^{X-1}\mid W_k]
    &=
    c_{\sss X}\alpha\Gamma(1-\alpha) W_k \theta(W_k)^{-(1-\alpha)}(1+o_\prob(1)),
    %\\& \geq c_x\alpha\Gamma(1-\alpha)\e^{-\underline a k} \theta(\e^{-\underline a k})^{-(1-\alpha)}(1+o_\prob(1)),
\end{aligned}
\end{equation}
with high probability and where the $o_\prob(1)$ is uniformly bounded.
Next, by the monotone density theorem \cite[Theorem 1.2.9]{Mikosch1999}, for $k\to\infty$,
\begin{equation}
\begin{aligned}
    \expec[X^2(1-W_k\theta(W_k))^{X-2}\mid W_k]
    &= c_{\sss X}\alpha\Gamma(1-\alpha)(1-\alpha) W_k \theta(W_k)^{-(2-\alpha)}(1+o_\prob(1))
    %\\& \leq c_x\alpha\Gamma(1-\alpha)\e^{-\bar a k} (1-\alpha)\theta(\e^{-\bar a k})^{-(2-\alpha)}(1+o_\prob(1)),
\end{aligned}
\end{equation}
with high probability  and where the $o_\prob(1)$ is uniformly bounded.
This implies that for $k\to\infty$ and for bounded error term $o_\prob(1)$,
\begin{equation}
\label{eq:third_rigime_eq1}
    \expec[\hat D_{W_k}\mid W_k]
    =
    \frac{1}{2}W_k
    (1-\alpha)(W_k\theta(W_k))^{-1}
    (1+o_\prob(1)).
\end{equation}
\noindent\textbf{Expected cumulative weight in the attached trees.}
We next investigate $\expec[|T_i^{W_k}|+1 \mid W_k]$,
where we recall that the offspring distribution of $T_i^{w_k}$ is given by $\tilde X_k$ and based on \eqref{eq:prereq_tildeX}, we find for $k\to\infty$ and with bounded error term $o_\prob(1)$,
\begin{equation}
\label{eq:offsring_tildeX_reg3}
    \expec[\tilde X_k\mid W_k]= 
    W_k\expec[X(1-W_k\theta(W_k))^{X-1} \mid W_k] =
    \alpha(1+o_\prob(1)).
\end{equation}
As the expected total progeny for an i.i.d. branching process with mean $\mu<1$ is known and equal to $1/(1-\mu)$, this implies then
\begin{equation}
    \label{eq:third_rigime_eq2}
    \expec[|T_i^{(W_k)}|+1 \mid W_k] =1+\frac{1}{1-\alpha}(1+o_\prob(1))
    =\frac{2-\alpha}{1-\alpha}(1+o_\prob(1)).
\end{equation}
\noindent\textbf{Finalising the expectation of $B_k$.}
Combining the above results we find with high probability for $k\to\infty$ with bounded error term $o_\prob(1)$,
\begin{equation}
    \expec[B_k\mid W_k]
    \leq 
   \Big( W_k+\frac{2-\alpha}{4} W_k\theta(W_k)^{-1} \Big)(1+o_\prob(1)).
\end{equation}
Based on Lemma \ref{lem:scaling_theta} we find that for $k>K$ that there exists a $c>0$ such that, %with high probability,
\begin{equation}
     \expec[B_k\mid W_k]
    \leq 
   \Big( \e^{-\underline ak}+\frac{2-\alpha}{4} \e^{-k(\underline a - \alpha/(1-\alpha)\bar a)}\Big)(1+o_\prob(1))
   \leq \e^{-c k}(1+o_\prob(1)),
\end{equation}
where the second inequality follows from the specific choices of $\underline a$ and $ \bar a $ as chosen in \eqref{eq:boundsW_k_reg3}. Moreover, as $W_k<1$ with probability 1, the $o_\prob(1)$ term is uniformly bounded, such that there exists a $C$ such that, for $k<K$,
\begin{equation}
    \expec[B_k] < C\e^{-ck}.
\end{equation}
This implies that $\expec[B_k]$ has an exponentially converging majorant and as $B_k$ is almost surely finite, we find
\begin{equation}
   \lim_{k\to\infty} \expec[\B_k]
    =\lim_{k\to\infty}\sum_{i=1}^k \expec[B_i]
    \leq \lim_{k\to\infty}\sum_{i=1}^k C\e^{-ci} < \infty
\end{equation}

\end{proof}

\subsection{Worked-out example}
\label{sec:worked out example}
We close the paper with a worked-out example. By chosing a specific generating function for the offspring distribution of $\T$,  we can explicitly calculate the cumulative weights of the $k$-cut IPC for $k\to\infty$. In this section we assume that, for $\alpha\in(0,1/2)$,
\begin{equation}
\label{eq:special_case_transform}
    \expec[s^X] = 1-(1-s)^\alpha.
\end{equation}
It is straightforward to verify that \eqref{eq:special_case_transform} is a generating function corresponding to a distribution that satisfies \eqref{eq:heavy_tail} with exponent $\alpha\in(0,1/2)$. For given $p\in(0,1)$, $\theta(p)$ can be expressed explicitly by \eqref{eq:eta(p)} as
\begin{equation}
\label{eq:theta_special_case}
    1-\theta(p)=
    \expec\big[(1-p\theta(p))^X\big]
    =1-(p\theta(p))^\alpha,
    \text{ so that }\
    \theta(p) = p^{\alpha/(1-\alpha)}.
\end{equation}
Furthermore, the Markov chain representation of $(W_k)_{k\geq 0}$ can be made explicit. By combining \eqref{eq:Wk_Mc} and \eqref{eq:special_case_transform} we find 
\begin{equation}
\begin{aligned}
    \prob(W_{k+1}=W_k\mid W_k ) &= 1-\alpha(p\theta
    (W_k))^{\alpha-1}\frac{\theta(W_k)}{\theta'(W_k)}
    =
    \alpha,\\ 
     \prob(W_{k+1}<u\mid W_k) &=\alpha(p\theta
    (W_k))^{\alpha-1}\frac{\theta(u)}{\theta'(W_k)}
    =(1-\alpha)\left(\frac{u}{W_k}\right)^{\alpha/(1-\alpha)}.
\end{aligned}
\end{equation}
This specifically implies that $W_k$ can be written as a product of i.i.d. random variables. Let $(P_i)_{i\geq1}$ and i.i.d. sequence such that $P_i$ equals $1$ with probability $\alpha$ and $\Unif[0,1]^{(1-\alpha)/\alpha}$ with probability $1-\alpha$, then
\begin{equation}
\label{eq:example_wk_product}
    W_k = W_0\prod_{i=1}^k 
    P_i,
\end{equation}
where we note that $\prob(W_0<p)=\theta(p)=p^{\alpha/(
1-\alpha)}$. It follows that many key components in the analysis of $\expec[\B_k]$ are explicit and indeed, we find the following explicit result:

\begin{example}[Expected cumulative weights]
\label{ex:expec cumulative weights}
    Consider offspring distribution $X$ with a generating function as given in \eqref{eq:special_case_transform} and let $\B_k$ as defined in \eqref{eq:def_B_k}. Then, for $\alpha\in(0,1/2)$,
    \begin{equation}
     \lim_{k\to\infty}  \expec[\B_k]=\frac{\alpha(2-5\alpha^2)}{4(1-\alpha)^2(1-2\alpha)}.
    \end{equation}
\end{example}

\section*{Acknowledgements}
The author thanks Remco van der Hofstad for his helpful comments and feedback. 
This work is supported by the Netherlands Organisation for Scientific Research (NWO) through Gravitation-grant NETWORKS-024.002.003.

\begin{appendices}

\section{Proof for tightness
}
\label{app:tightness}
We use the Arzel\'a-Ascoli characterisation for tightness \cite[Theorem 13.2]{Billingsley1999}. This states that tightness is equivalent to the following two conditions (see also \cite[Appendix B]{Gundlach2022}:
\begin{description}
\item[Tightness of the supremum.] For every $\eta>0$, there exists an $a>0$ such that 
\begin{equation}
     \underset{k\to\infty}{\emph{limsup}}
    \prob(\sup_{t\in [\vep,R]} |X_k(t)| >a)<
    \eta.
\end{equation}
		
\item[Modulus of continuity.] For every $r>0$ and $\eta>0$, there exists a $\delta>0$ such that
\begin{equation}
\label{eq:modulus o c}
    \underset{k\to \infty}{\emph{limsup}}
    \prob(\bar{w}_{X_k}(\delta)>r)<\eta,
\end{equation}
where, for a process $x=(x(t))_{t\in [\vep,R]}$, $\bar{w}_x(\delta)$ is the c\`adl\`ag modulus of $x$ given by
\begin{equation}
\inf
 \limits_{(t_i)_{i\geq1}\in P_\delta}\
\max
 \limits_{i\in[1,j]}\
\sup
  \limits_{s,t\in[t_i,t_{i+1}]}
|x(t)-x(s)|,
\end{equation}
and $P_\delta$ is the set of all partitions of the interval $[\varepsilon,R]$, where the increment lengths are at most $\delta$. 
\end{description}

In the following, we use that $k^{-\gamma}M_{\lceil kt\rceil}$ is tight. Combining this with the observation in \eqref{eq:stochdomB} show the condition for tightness of the supremum follows directly, for given $\eta$ there exists an $a$ such that
\begin{equation}
\underset{k\to\infty}{\emph{limsup}}
		\prob(\sup_{t\in [\vep,R]} |\B_{\lceil kt\rceil}(t)| >a)\leq 
 \underset{k\to\infty}{\emph{limsup}}
		\prob(\sup_{t\in [\vep,R]} |M_{\lceil kt\rceil}(t)| >a)<
		\eta.
\end{equation}
Next, for the condition for the modulus of continuity, we observe again that
\begin{equation}
\label{eq:modulus o c: lem}
    |\B_{\lceil kt_2\rceil}(t)
    -\B_{\lceil kt_1\rceil}(t)
    |
    =
    \sum_{i=\lceil kt_1\rceil+1}^{\lceil kt_2\rceil} \sum_{e\in E(\H_i)} w(e)
    \preceq
    \sum_{i=\lceil kt_1\rceil}^{\lceil kt_2\rceil}H_i =
    |M_{\lceil kt_2\rceil}(t)
    -M_{\lceil kt_1\rceil}(t)
    |.
\end{equation}
By tightness of $(k^{-\gamma}M_{\lceil kt\rceil})_{t>0}$,  there exists a suitable partition $P_\delta$ such that the \eqref{eq:modulus o c} is satisfied for $(k^{-\gamma}M_{\lceil kt\rceil})_{t>0}$.  The same partition can be used in combination with \eqref{eq:modulus o c: lem} to show that $P_\delta$ is also a suitable partition to show \eqref{eq:modulus o c} for $(k^{-\gamma}\B_{\lceil kt\rceil})_{t>0}$.

\section{Details on the worked-out example
}
We next provide the details of Example \ref{ex:expec cumulative weights}
    To derive $\expec[\B_k]$ we start with deriving $\expec[B_k\mid W_k]$ that we deconstruct in \eqref{eq:expec_B_k}. 
    Then, we extend the result to the unconditional expectation.
    
    \paragraph{The conditional expectation of the factors in \eqref{eq:expec_B_k}.}
    Firstly, based on the notion in \eqref{eq:example_wk_product} the expected weights on the backbone are easily derived by
    \begin{equation}
    \label{eq:worked_example_1}
    \begin{aligned}
         \expec[\beta_{k+1}\mid W_k] &= \alpha W_k + (1-\alpha)\expec[\Unif[0,W_k]]
         =
         \frac{1}{2}
         ( 1+\alpha)
        W_k.
    \end{aligned}
    \end{equation}
 For the degrees of the backbone vertices we find based on \eqref{eq:bbweights_reg3}
 \begin{equation}
    \expec[\hat D_{W_k}\mid W_k]
    =\frac{1}{2}W_k\frac{\expec[X^2(1-W_k\theta(W_k))^{X-1}\mid W_k]}{\expec[X(1-W_k\theta(W_k))^{X-1}\mid W_k]}=
    \frac{ \alpha(1-\alpha)(W_k\theta(W_k))^{\alpha-2}(1-W_k\theta(W_k))}{2\alpha(W_k\theta(W_k))^{\alpha-1}}W_k
 \end{equation}
 Based on \eqref{eq:theta_special_case},
 we can simply this to
 \begin{equation}
 \label{eq:worked_example_2}
     \expec[\hat D_{W_k}\mid W_k]
    =\frac{1}{2}(1-\alpha)(1-W_k^{1/(1-\alpha)})
    W_k^{-\alpha/(1-\alpha)}.
 \end{equation}
 For the sizes of the attached forests, we recall that $\tilde X_k$ denotes the offspring distribution of a tree that is restricted to weights below $W_k$ that is conditionally finite. By
 \eqref{eq:offsring_tildeX_reg3} and \eqref{eq:theta_special_case},
 \begin{equation}
     \expec[\tilde X_k\mid W_k]=
     W_k\expec[X(1-W_k\theta(W_k))^{X-1}\mid W_k]= W_k\alpha(W_k\theta(W_k))^{\alpha-1}=\alpha. 
 \end{equation}
  Therefore, we can also conclude that 
  \begin{equation}
  \label{eq:worked_example_3}
      \expec\Big[|T_1^{W_k}|+1\ \Big| \ W_k\Big]
      =\frac{1}{1-\alpha}+1 = \frac{2-\alpha}{1-\alpha}.
  \end{equation}
 Substituting these \eqref{eq:worked_example_1}, \eqref{eq:worked_example_2} and \eqref{eq:worked_example_3} in 
 \eqref{eq:expec_B_k} gives
  \begin{equation}
  \label{eq:worked_example_4}
  \begin{aligned}
      \expec[B_k\mid W_k]
     & =
      \frac{1}{2}{(1+\alpha)W_k}
      +
      \frac{1}{4}(2-\alpha)(1-W_k^{\alpha/(1-\alpha)}) W_k^{(1-2\alpha)/(1-\alpha)}\\
      &=
       \frac{1}{2}{(1+\alpha)W_k}
      +
      \frac{1}{4}(2-\alpha)(1-W_k^{\alpha/(1-\alpha)}) W_k^{(1-2\alpha)/(1-\alpha)}
      .
    \end{aligned}
  \end{equation}
 Therefore by \eqref{eq:def_texttt_B_k} and by rearranging and simplifying the terms in \eqref{eq:worked_example_4},
 \begin{equation}
 \label{eq:worked_example_5}
     \expec[\B_k\mid W_k]=
     \sum_{i=0}^k
     \frac{3}{4}\alpha W_i
    +\frac{1}{4}(2-\alpha) W_i^{(1-2\alpha)/(1-\alpha)}.
 \end{equation}
 \paragraph{Extension to the unconditional expectation.}
 In the final step of this example we take the expectation of the expression in \eqref{eq:worked_example_5}. In order to do so we first derive the expectation of $W_i^{(1-2\alpha)/(1-\alpha)}$  and find
 \begin{equation}
 \begin{aligned}
     \expec[W_i^{(1-2\alpha)/(1-\alpha)}]&=
     \expec\Big[W_0^{(1-2\alpha)/(1-\alpha)} \prod_{j=1}^i (P_j)^{(1-2\alpha)/(1-\alpha)}\Big]
    \\& =\expec[W_0^{(1-2\alpha)/(1-\alpha)}] 
     \Big(
    \alpha +(1-\alpha)\expec\Big[\Unif[0,1]^{(1-2\alpha)/\alpha}\Big]
     \Big)^i.
    \end{aligned}
 \end{equation}

We note here that 
\begin{equation}
    \expec\Big[\Unif[0,1]^{(1-2\alpha)/\alpha}\Big] = 
    \int_{0}^{1}
    u^{(1-2\alpha)/\alpha}\dif u
    =
    \frac{\alpha}{1-\alpha}.
\end{equation}
and by recalling that $W_0$ has a density $\theta'(u)$ we find 
\begin{equation}
    \expec[W_0^{(1-2\alpha)/(1-\alpha)}]
    =
    \int_0^1 u^{(1-2\alpha)/(1-\alpha)}
    \frac{\alpha}{1-\alpha}u^{\alpha/(1-\alpha)-1} \dif u 
    = \frac{\alpha}{1-\alpha}.
\end{equation}
Therefore, 
\begin{equation}
    \expec[W_i^{(1-2\alpha)/(1-\alpha)}]=
    \frac{\alpha}{1-\alpha} (2\alpha)^i.
\end{equation}
By a similar approach we also derive that
\begin{equation}
    \expec[W_i]
    =
    \alpha (\alpha(2-\alpha))^i.
\end{equation}
Therefore, we conclude that by taking expectations on both side of \eqref{eq:worked_example_5}
\begin{equation}
\begin{aligned}
        \expec[\B_k] &=
   \Big( \frac{1}{4}\sum_{i=1}^k
   3\alpha^2
    (\alpha(2-\alpha))^i
    +\frac{\alpha(2-\alpha)}{1-\alpha}(2\alpha)^i\Big).
\end{aligned}
\end{equation}
Taking a limit for $k\to\infty$ then concludes
\begin{equation}
    \begin{aligned}
        \lim_{k\to\infty}
        \expec[\B_k]
            &=
     \frac{\alpha}{4}\Big(
     \frac{3\alpha}{1-\alpha(2-\alpha)}
     +
    \frac{(2-\alpha)}{1-\alpha}\frac{1}{1-2\alpha}
    \Big)
    \\&
    =\frac{\alpha}{4}
    \Big( \frac{3\alpha}{(1-\alpha)^2}+\frac{2-\alpha}{(1-2\alpha)(1-\alpha)}\Big)\\&=
    \frac{\alpha(2-5\alpha^2)}{4(1-\alpha)^2(1-2\alpha)}.
    \end{aligned}
\end{equation}
    
\end{appendices}

\bibliographystyle{abbrv}
\bibliography{Zbib}

\end{document}